\documentclass[12pt,oneside]{amsart}
\usepackage{amsmath,amssymb,amsthm}
\usepackage{xcolor}
\usepackage{comment}
\usepackage{geometry} 
\usepackage{appendix}
\usepackage{enumerate}
\usepackage{listings}
\usepackage[colorlinks=true, pdfstartview=FitH, linkcolor=blue, citecolor=blue, urlcolor=blue]{hyperref}

\newtheorem{thm}{Theorem}[section]
\newtheorem{lem}[thm]{Lemma}

\newtheorem{defn}[thm]{Definition}

\newtheorem{cor}[thm]{Corollary}
\newtheorem{rem}[thm]{Remark}

\newtheorem{conj}[thm]{Conjecture}

\newcommand{\CC}{\mathcal{C}} 
\newcommand{\FF}{\mathcal{F}}

\newcommand{\F}{\mathbb{F}}

\begin{document}

\title{On maximal cliques of Cayley graphs over fields}
\author{Chi Hoi Yip}
\address{Department of Mathematics \\ University of British Columbia \\ 1984 Mathematics Road \\ Canada V6T 1Z2}
\email{kyleyip@math.ubc.ca}
\date{\today}

\begin{abstract}
We describe a new class of maximal cliques, with a vector space structure,  of Cayley graphs defined on the additive group of a field. In particular, we show that in the cubic Paley graph with order $q^3$, the subfield with $q$ elements forms a maximal clique. Similar statements also hold for quadruple Paley graphs and Peisert graphs with quartic order.
\end{abstract}
\maketitle

\section{Introduction}

%In this paper, we will focus on maximal cliques of a Cayley graph defined on the additive group of a field. Throughout the paper, we let $p$ be an odd prime, $q$ a power of $p$, and $\F_q$ the finite field with $q$ elements. We first recall some basic terminologies from graph theory.

For an abelian group $G$ and a connection set $S \subset G$ with $S=-S$, the {\em Cayley graph} $\operatorname{Cay}(G;S)$ is
the graph whose vertices are elements of $G$, such that two vertices $g$ and $h$ are adjacent if and only if $g-h \in S$. The assumption that $S=-S$ guarantees the graph is undirected.

A {\em clique} in a graph $X$ is a subgraph of $X$ that is a complete graph.  A {\em maximum clique} is a clique with the maximum size, while a {\em maximal clique} is a clique where one cannot add another vertex to it and still have a clique. For a graph $X$, the {\em clique number} of $X$, denoted $\omega (X)$, is the size of a maximum clique of $X$. For a given graph $X$, it is often interesting to estimate its clique number and classify maximum cliques as well as maximal cliques. It is not surprising if there exists a maximal clique in $X$ which is not maximum, but in general, it is challenging to construct such examples explicitly if $X$ has a large order. A simple reason is that finding the clique number of a graph is extremely difficult; to be precise, NP-complete \cite{RK}. We refer to a recent paper by Godsil and Rooney \cite{GR} for a discussion on the hardness of the computation on the clique number of Cayley graphs. 

However, for a given Cayley graph, we might be able to deduce some information on its cliques based on its algebraic structure. Using tools from additive combinatorics and random graph theory, Green and Morris \cite{Green, GM} showed that the clique number of almost all Cayley graphs defined on a cyclic group $G$ is $O(\log |G|)$. This suggests that if a Cayley graph is exceptional in the sense that the clique number is significantly larger than the logarithm scale, we might expect the graph has some additional proprieties. We will focus on this type of exceptional Cayley graphs defined on the additive group of a field. Throughout the paper, we let $p$ be an odd prime, $q$ be a power of $p$, and $\F_q$ denote the finite field with $q$ elements.

Our first result in this paper is the following theorem regarding a new family of maximal cliques in certain Cayley graphs. It states that there is a maximal clique with a vector space structure under certain assumptions.

\begin{thm}\label{Cayleythm}
Let $F$ be a field with the additive group $F^+$, and let $S$ be a connection set such that $S \subset F \setminus \{0\}$, $-1 \in S$,  and $S$ is closed under multiplication. Assume that $K$ is a proper subfield of $F$, such that $[F:K]<\infty$ and $K$ forms a clique in the Cayley graph $X=\operatorname{Cay}(F^+;S)$. Then there is a subspace $V$ of $F$ over $K$, with $K \subset V$, such that $V$ forms a maximal clique in $X$. If we assume the axiom of choice, then the assumption that $[F:K]$ is finite can be dropped.
\end{thm}

Two particular nice and well-studied classes of Cayley graphs are Paley graphs and Peisert graphs. It is well known that Paley graphs and Peisert graphs are strongly regular, symmetric, and self-complementary; see for example \cite{WP2}, \cite[Chapter 2]{thesis} and \cite[Chapters 2 and 3]{NM}. (Generalized) Paley graphs and Peisert graphs have many applications in combinatorics, number theory,  design theory, and coding theory; see for example \cite{GJ,KR,LJ}.

Paley graphs are defined on the additive group $\F_q^+$ of a finite field $\F_q$. Suppose $p$ is a prime, such that $q=p^s \equiv 1 \pmod{4}$. The {\em Paley graph} on $\F_q$, denoted $P_q$, is the graph whose vertices are the elements of $\F_q$, such that two vertices are adjacent if and only if the difference of the two vertices is a square in $\F_q$. Note that $P_q=\operatorname{Cay}(\F_q^+;(\F_q^*)^2)$. 

Similarly one can define generalized Paley graphs. They were first introduced by Cohen \cite{SC} in 1988, and reintroduced by Lim and Praeger \cite{LP} in 2009. Let $d>1$ be a positive integer. The {\em $d$-Paley graph} on $\F_q$, denoted $GP(q,d)$, is the graph whose vertices are the elements of $\F_q$, where two vertices are adjacent if and only if the difference of the two vertices is a $d$-th power of $x$ for some $x \in \F_q^*$. It is standard (see for example \cite[Section 4]{SC}) to further assume that $q \equiv 1 \pmod {2d}$. Note that $2$-Paley graphs are just the standard Paley graphs. $3$-Paley graphs are also called {\em cubic Paley graphs}, $4$-Paley graphs are also called {\em quadruple Paley graphs} \cite{WA}. Also note that generalized Paley graphs are in fact Cayley graphs. We have $GP(q,d)=\operatorname{Cay}(\F_q^+;(\F_q^*)^d)$, where $(\F_q^*)^d$ is the set of $d$-th powers in $\F_q^*$. 

In the literature \cite{BDR,SC,HP,Yip,Yip2}, the trivial upper bound on $\omega\big(GP(q,d)\big)$ is given by $\sqrt{q}$, and a simple proof can be found in \cite[Lemma 5.2]{Yip2}. We will refer to this trivial upper bound constantly in Section 3. When $q$ is a square, it is easy to show that $\omega(P_q)=\sqrt{q}$; moreover, Blokhuis \cite{AB} classified all maximum cliques in $P_q$ and showed that the only maximum clique containing $0,1$ is the subfield $\F_{\sqrt{q}}$. Later Sziklai \cite{PS} generalized Blokhuis's proof and extended this characterization of maximum cliques to certain generalized Paley graphs.

For a Paley graph with square order, explicit families of maximal cliques that are not maximum are known by the independent works of Baker et al. \cite{BEHW}, and Goryainov et al. \cite{GKSV}. 

\begin{thm}[\cite{BEHW}, \cite{GKSV}] \label{max2}
Let $q$ be an odd prime power. In the Paley graph with order $q^2$, there is a maximal clique of size $\frac{1}{2}(q+1)$ or $\frac{1}{2}(q+3)$, accordingly as $q \equiv 1\pmod 4$ or $q \equiv 3 \pmod 4$.
\end{thm}

%When $q$ is a square, by considering the subfield $\F_{\sqrt{q}}$, one can show the trivial upper bound is tight \cite{BDR} for Paley graphs. 

In general, we have the following lower bounds based on subfield construction.

\begin{thm}[\cite{BDR}]\label{t4}
Let $d$ be a positive integer greater than $1$. Let $q \equiv 1 \pmod {2d}$ be a power of a prime $p$. If $k$ is an integer such that $d \mid \frac{q-1}{p^k-1}$, then the subfield $\F_{p^k}$ forms a clique in $GP(q,d)$. In particular, if $r$ is the largest integer such that $d \mid \frac{q-1}{p^r-1}$, then $\omega\big(GP(q,d)\big) \geq p^r$.
\end{thm}

Based on Theorem \ref{t4}, it is natural to come up with the following conjecture. 

\begin{conj}\label{maxconj}
Let $d$ be a positive integer greater than $1$. Let $q \equiv 1 \pmod {2d}$ be a power of a prime $p$, and let $r$ be the largest integer such that $d \mid \frac{q-1}{p^r-1}$. Then the subfield $\F_{p^r}$ forms a maximal clique in $GP(q,d)$.
\end{conj}

Note that when $q=p^{2r}$ and $d \mid (p^r+1)$, Conjecture \ref{maxconj} holds trivially: the trivial upper bound on the clique number is $\sqrt{q}=p^r$, and Theorem \ref{t4} implies that $\F_{\sqrt{q}}$ forms a maximal clique.

We will apply Theorem \ref{Cayleythm} to generalized Paley graphs and verify that Conjecture \ref{maxconj} is true for certain cubic Paley graphs and quadruple Paley graphs. In particular, for cubic Paley graphs (with cubic order) and quadruple Paley graphs (with quartic order), we manage to construct explicit maximal cliques. 

\begin{thm}\label{max3}
Let $r$ be a positive integer. If $p$ is a prime such that $q=p^{3r} \equiv 1 \pmod 6$, then $\F_{p^r}$ is a maximal clique in the cubic Paley graph $GP(q,3)$.
\end{thm}

\begin{thm}\label{max4}
Let $r$ be a positive integer. If $p$ is an odd prime and $q=p^{4r}$, then $\F_{p^r}$ is a maximal clique in the quadruple Paley graph $GP(q,4)$.
\end{thm}

By Theorem \ref{t4} and Theorem \ref{max3}, we can deduce that: if $p^r \equiv 5 \pmod 6$, then in the cubic Paley graph $GP(p^{6r},3)$, $\F_{p^{2r}}$ forms a maximal clique, while the clique number of the graph is $p^{3r}$ and one maximum clique is given by $\F_{p^{3r}}$. This means we have explicitly constructed a maximal clique in a cubic Paley graph, which is not maximum. This phenomenon shares some similarity with Theorem \ref{max2}.

We will introduce Peisert graphs formally in Section 4. Peisert graphs are Cayley graphs defined on the additive group of a finite field with square order, and Peisert graphs are similar to Paley graphs in many aspects (see the discussion on \cite{KP}). However, little is known about the cliques of Peisert graphs. The next result provides a connection between maximal cliques and the clique number of a Peisert graph with quartic order.

\begin{thm}\label{maxP*}
Let $r$ be a positive integer, $p$ be a prime such that $p \equiv 3 \pmod 4$, and  $q=p^{4r}$.  If $\F_{p^r}$ is not a maximal clique in the Peisert graph $P^*_q$, then $\omega(P^*_q)= \sqrt{q}$; moreover, there exists $h \in \F_q \setminus \F_{p^r}$, such that $\F_{p^r} \oplus h\F_{p^r}$ forms a maximum clique, and $\{1,h,g^2,g^2h\}$ forms a basis of $\F_q$ over $\F_{p^r}$ for any primitive root $g$ in $\F_q$.
\end{thm}

The paper is organized as follows. We will prove Theorem \ref{Cayleythm} in Section 2. Then, in Section 3, we will apply Theorem \ref{Cayleythm} to generalized Paley graphs and prove Theorem \ref{max3} and Theorem \ref{max4}. We will introduce Peisert graphs and discuss maximal cliques of Peisert graphs in Section 4, where Theorem \ref{Cayleythm} is not directly applicable. We will then show Theorem \ref{maxP*} and state some conjectures on the cliques of Peisert graphs.

\section{Proof of Theorem \ref{Cayleythm}}
We will prove Theorem \ref{Cayleythm} assuming the axiom of choice, which is equivalent to Zorn's lemma. When $[F:K]<\infty$, Theorem \ref{Cayleythm} follows from Lemma \ref{Caylem}. We first recall some basic definitions.

\begin{defn}
Let $A$ be partially ordered by $\leq$.
\begin{itemize}
    \item A subset $B \subset A$ is a chain if for all $x, y \in B,$ either $x \leq y$ or $y \leq x$.
    \item An upper bound for $B \subset A$ is an element $u \in A$ such that $b \leq u$ for all $b \in B$.
    \item  A maximal element of $A$ is an element $m \in A$ such that $m \leq x$ implies $m=x$ for any $x \in A$.
\end{itemize}
\end{defn}

\begin{thm} [Zorn's lemma] 
If $A$ is a nonempty partially ordered set in which every chain has an upper bound, then $A$ has a maximal element.
\end{thm}

The following lemma is the key to prove Theorem \ref{Cayleythm}.

\begin{lem}\label{Caylem}
Let $F,S,K,X$ be defined as in Theorem \ref{Cayleythm}. Let $V$ be a clique in $X$ such that $V$ is a subspace of $F$ over $K$, with $K \subset V$. If $V$ is not a maximal clique in $X$, then there is $g \in F \setminus V$, such that $V \oplus gK$ is a clique in $X$. 
\end{lem}

\begin{proof}
Assume that $V$ is not a maximal clique. Then there is $g \notin V$, such that $\{g\} \cup V$ forms a larger clique, i.e. for any $v \in V$, $g+v=g-(-v) \in S$. The assumption that $V$ is a clique in $X$ implies that for any distinct $u,v \in V$, $u-v \in S$. 

Let $a,b \in K$ and $u,v \in V$, we consider whether the two distinct vertices $u+ga$ and $v+gb$ are adjacent. We have the difference
$
(u+ga)-(v+gb)=(u-v)+g(a-b).
$
\begin{itemize}
    \item If $a=b$, then $(u+ga)-(v+gb)=u-v \in S.$
    \item If $a \neq b$, then $(u+ga)-(v+gb)=(a-b) \big(g+ (a-b)^{-1} (u-v)\big)$. Note that $a,b \in K$ implies $a-b \in S$, and $(a-b)^{-1}(u-v) \in V$ implies $g+ (a-b)^{-1} (u-v) \in S$. Since $S$ is closed under multiplication, we obtain that $(u+ga)-(v+gb) \in S$.
\end{itemize}
Therefore, $V \oplus gK$ forms a clique in $X$.
\end{proof}

Now we are ready to prove Theorem \ref{Cayleythm}.

\begin{proof}[Proof of Theorem \ref{Cayleythm}]
Let $\CC$ be the set of cliques in the Cayley graph $X$, which are subspaces of $F$ over $K$. Then $\CC$ is partially ordered by the set inclusion and the assumption that the proper subfield $K$ forms a clique implies that $\CC$ is nonempty.

Let $\FF \subset \CC$ be a chain, we claim that $M=\cup_{C \in \FF} C$ is an upper bound for $\FF$. To prove that, it suffices to show $M \in \CC$, i.e. $M$ is a clique, with vector space structure. It is easy to verify that $M$, as the union of vector spaces in the chain $\FF$, is a vector space over $K$. Let $a, b \in M$, then $a \in C_1, b \in C_2$ for some $C_1, C_2 \in \FF$. Since $\FF$ is a chain, we may as well assume $C_1 \subset C_2$, then $a,b \in C_2$, and thus $a-b \in S$. This shows that $M$ is a clique.

Therefore, by Zorn's lemma, $\CC$ has a maximal element, say $C$. If $C$ is not a maximal clique in $X$, then by Lemma \ref{Caylem}, there is $C' \in \CC$, such that $C$ is a proper subset of $C'$, which contradicts the maximality of $C$ in $\CC$. Therefore, $C$ is indeed a maximal clique in $X$. 
\end{proof}

\begin{rem}
Note that in the proof of Theorem \ref{Cayleythm} and Lemma \ref{Caylem}, we did not use the fact that the multiplication in the field is commutative. By slightly modifying the statement and the proof of Lemma \ref{Caylem}, it is straightforward to generalize Theorem \ref{Cayleythm} to the setting of division rings. Recall that division rings are rings in which every nonzero element has a multiplicative inverse, and Wedderburn's little theorem states that all finite division rings are fields. The notion of vector spaces over a field will be replaced by modules over a division ring accordingly. 
\end{rem}

\section{Maximal cliques of generalized Paley graphs}

In this section, we apply Theorem \ref{Cayleythm} to generalized Paley graphs. In the following discussion, we use the following standard notation: for a vector space $V$ over a field $K$, we use $\operatorname{dim}_{K} V$ to denote the dimension of $V$ over $K$.

Note that if the field $F$ is given by the finite field $\F_q$, and $S \subset \F_q^*$ is closed under multiplication, then $S$ must be also closed under inverse, and thus $S$ must be a multiplicative subgroup of $\F_q^*$. Recall that generalized Paley graphs are Cayley graphs defined over a finite field $\F_q$, with the connection set being a nontrivial multiplicative subgroup of $\F_q^*$. Therefore, in the case that the field $F$ is finite, the Cayley graph $X$ defined in Theorem \ref{Cayleythm} is exactly given by a generalized Paley graph.

\begin{cor}\label{corGP}
Let $t,d,s$ be positive integers such that $d>1$. Let $p$ be a prime such that $q=p^s \equiv 1 \pmod {2d}$. Suppose there is a proper subfield $K$ of $\F_q$ with $|K|=p^t$, such that $K$ forms a clique in the generalized Paley graph $GP(q,d)$. Then there is a subspace $V$ of $\F_q$ over $K$, such that $K \subset V$, $\operatorname{dim}_{K} V \leq \frac{s}{2t}$, and $V$ forms a maximal clique in $GP(q,d)$. 
\end{cor}

\begin{proof}
Theorem \ref{Cayleythm} implies the existence of such $V$ without the restriction on $\operatorname{dim} V$. Recall that the trivial upper bound on the clique number is $\sqrt{q}$. Thus, $\operatorname{dim}_{\F_p} V= t\operatorname{dim}_K V \leq s/2$.
\end{proof}

An immediate corollary of Corollary \ref{corGP} is the following.

\begin{cor}\label{corrr}
Let $r,d$ be positive integers such that $d>1$. If $p$ is a prime such that $q=p^{dr} \equiv 1 \pmod {2d}$ and $d \mid \frac{q-1}{p^r-1}$, then either $\F_{p^{r}}$ is a maximal clique in $GP(q,d)$, or $\omega\big(GP(q,d)\big) \geq q^{2/d}$.
\end{cor}
\begin{proof}
By Theorem \ref{t4}, $\F_{p^{r}}$ is a clique in $GP(q,d)$. By Corollary \ref{corGP}, there is a maximal clique $V$ in $GP(q,d)$, such that $V$ is a vector space over $\F_{p^r}.$ If $\operatorname{dim}_{\F_{p^r}} V=1$, then $V=\F_{p^r}$ is a maximal clique; otherwise, $\operatorname{dim}_{\F_{p^r}} V \geq 2$ and thus $\omega\big(GP(q,d)\big) \geq |V| \geq p^{2r}=q^{2/d}$.
\end{proof}

Now it is straightforward to prove Theorem \ref{max3}.
\begin{proof}[Proof of Theorem \ref{max3}]
Since $q=p^{3r} \equiv 1 \pmod 6$, we have $p^r \equiv 1 \pmod 3$ and $\frac{p^{3r}-1}{p^r-1}= 1+p^r+p^{2r} \equiv 0 \pmod 3$. By Corollary \ref{corrr}, either $\F_{p^{r}}$ is a maximal clique in the cubic Paley graph $GP(q,3)$ , or $\omega\big(GP(q,3)\big) \geq q^{2/3}$. However, the second possibility is ruled out by the trivial upper bound $\omega\big(GP(q,3)\big) \leq \sqrt{q}$.
\end{proof}

For a quadruple Paley graph of quartic order, Corollary \ref{corrr} implies the following.  Note that we do not need to assume $q=p^{4r} \equiv 1 \pmod 8$ since it is always guaranteed for an odd prime $p$.
\begin{cor} \label{cor4}
Let $r,d$ be a positive integer. If $p$ is an odd prime and $q=p^{4r}$, then either $\F_{p^{r}}$ is a maximal clique in the quadruple Paley graph $GP(q,4)$, or $\omega\big(GP(q,4)\big) =\sqrt{q}$.
\end{cor}
\begin{proof}
Note that $\frac{p^{4r}-1}{p^r-1}=1+p^r+p^{2r}+p^{3r} \equiv 0 \pmod 4$ for both the cases $p^r \equiv 1 \pmod 4$ and $p^r \equiv 3 \pmod 4$. The statement of the corollary follows form Corollary \ref{corrr} and the $\sqrt{q}=q^{2/4}$ trivial upper bound on $\omega\big(GP(q,4)\big)$.
\end{proof}
Therefore, to prove Theorem \ref{max4}, it suffices to show that the trivial upper bound $\sqrt{q}$ cannot be attained. There are some recent improvements on the upper bound of the clique number of a generalized Paley graph using polynomial methods and basic number theory when $q$ is a non-square \cite{BSW, HP, Yip, Yip2} and some methods could be extended to the case that $q$ is a square \cite{Yip2}. In particular, we need the following two results from \cite{Yip2}.
\begin{thm}[{\cite[Theorem 5.8]{Yip2}}] \label{t5}
Let $d$ be a positive integer larger than 1. Let $p$ be a prime and $q$ be a power of $p$ such that $q \equiv 1 \pmod{2d}$. If $2 \leq n\leq N=\omega\big(GP(q,d)\big)$ satisfies
$
\binom{n-1+\frac{q-1}{d}}{\frac{q-1}{d}}\not \equiv 0 \pmod p,
$
then $(N-1)n \leq \frac{q-1}{d}$.
\end{thm}

\begin{thm}[{\cite[Theorem 5.12]{Yip2}}]\label{t7}
Let $d$ be a positive integer at least $3$ and $p$ be an odd prime such that $d \mid (p-1)$. If $q$ is an even power of $p$, then $\omega\big(GP(q,d)\big) < \sqrt{\frac{q}{d}} \big(1+\frac{1}{2\sqrt{d}}+\frac{1}{8d}\big)+1.$
\end{thm}

In \cite[Section 5.3]{Yip2}, the author described a general approach to improve the trivial upper bound on the clique number of generalized Paley graphs. The author illustrated the process by improving the $\sqrt{q}$ bound for cubic Paley graphs. The same method can be used to show there is a constant $c<1$, such that $\omega(GP(q,4))<c\sqrt{q}$ for any $q=p^{4r}$. In the case $p \equiv 1 \pmod 4$, this follows from Theorem \ref{t7}. In the case $p \equiv 3 \pmod 4$, in the following lemma we prove a weaker bound; the proof is similar to the proof of \cite[Lemma 5.11]{Yip2}.

\begin{lem}\label{cor44}
Let $r$ be a positive integer, and $p$ a prime such that $p \equiv 3 \pmod 4$. If $q=p^{4r}$, then $\omega\big(GP(q,4)\big)\leq \sqrt{q}-1$.
\end{lem}

\begin{proof}
Suppose $\omega\big(GP(q,4)\big)>\sqrt{q}-1$. Together with the trivial upper bound $\sqrt{q}$, we have $N=\omega\big(GP(q,4)\big)=\sqrt{q}=p^{2r}$. Note that the base-$p$ representation of $\frac{q-1}{4}$ is $$\frac{q-1}{4}=\bigg(\frac{p-3}{4},\frac{3p-1}{4},\frac{p-3}{4},\frac{3p-1}{4}, \ldots, \frac{p-3}{4},\frac{3p-1}{4}\bigg)_p.$$
We can take $n-1=\frac{3p-1}{4}p^{2r-1}$ so that $n \leq p^{2r}=N$. Lucas's theorem implies that $$\binom{n-1+\frac{q-1}{4}}{\frac{q-1}{4}} \equiv \binom{\frac{p-3}{4}+\frac{3p-1}{4}}{\frac{p-3}{4}} \equiv \binom{p-1}{{\frac{p-3}{4}}}\not \equiv 0 \pmod p.
$$
By Theorem \ref{t5}, we get $(N-1)n \leq \frac{q-1}{4}$. Therefore, $$(p^{2r}-1)\frac{3p-1}{4}p^{2r-1} \leq \frac{p^{4r}-1}{4}.$$
It follows that
$
3p^{2r}-p^{2r-1} \leq p^{2r}+1,
$
which is impossible. 
\end{proof}

Now we are ready to prove Theorem \ref{max4}.
\begin{proof}[Proof of Theorem \ref{max4}]
By Corollary \ref{cor4}, it suffices to show that the trivial upper bound $\sqrt{q}$ on the clique number of $GP(q,4)$ cannot be attained. This follows immediately from Theorem \ref{t7} in the case $p \equiv 1 \pmod 4$, and Lemma \ref{cor44} in the case $p \equiv 3 \pmod 4$.
\end{proof}

\section{Maximal cliques of Peisert graphs}
In this section, we discuss the maximal cliques of Peisert graphs.

It is well-known that Paley graphs are self-complementary and symmetric. In \cite{WP2}, Peisert discovered a new infinite family of self-complementary symmetric graphs, called $P^*$-graphs. Later people refer to this new family of graphs as Peisert graphs. In fact, Peisert \cite{WP2} showed that apart from an exceptional graph with $23^2$ vertices, Paley graphs and Peisert graphs are the only self-complementary symmetric graphs. Similar to Paley graphs, Peisert graphs are defined on finite fields. 

The {\em Peisert graph} of order $q=p^r$, where $p$ is a prime such that $p \equiv 3 \pmod 4$ and $r$ is even, denoted $P^*_q$, is defined to be the graph with vertices in $\F_{q}$, such that two vertices are adjacent if their difference belongs to the set $$M_q=\{g^j: j \equiv 0,1 \pmod 4\},$$ where $g$ is a primitive root of the field $\F_q$. It is easy to see that the definition does not depend on the choice of the primitive root $g$. In the language of Cayley graphs, we have $P^*_q=\operatorname{Cay}(\F_{q}^+; M_q)$. Note that $M_q$ is not closed under multiplication since $g \cdot g=g^2 \notin M_q$, so we cannot apply Theorem \ref{Cayleythm} directly to Peisert graphs.

Kisielewicz and Peisert \cite{KP} extended the known results of Paley graphs to Peisert graphs. Nevertheless, there is no known improvement to the trivial upper bound $\sqrt{q}$ on its clique number at all. It seems those polynomial methods designed for Paley graphs could not be adapted to work for Peisert graphs directly. The only advantage we can take from the known results on Paley graphs is the following: a Peisert graph $P_q^*$ contains the quadruple Paley graph $GP(q,4)$ as a subgraph, so Peisert graphs have richer cliques than their corresponding quadruple Paley graphs.

The following lemma gives the square root trivial upper bound on the clique number of a Peisert graph. 

\begin{lem}\label{triv}
If $q=p^r$, where $p \equiv 3 \pmod 4$ and $r$ is even, then $\omega(P^*_q) \leq \sqrt{q}$. Moreover, the equality holds if and only if $\F_q=C+g^2C=\{u+g^2v: u,v \in C\}$ for each maximum clique $C$ in $P^*_q$ and each primitive root $g$ of $\F_q$.
\end{lem}
\begin{proof}
Let $N=\omega\big(P^*_q)$ and let $C=\{v_1, v_2, \ldots, v_N\} \subset \F_q$ be a maximum clique in $P^*_q$. Let $g$ be a be a primitive root of $\F_q^*$, and consider the set $W=\{v_i+g^2v_j: 1 \leq i,j \leq N\}$. 
Note that if $v_i+g^2v_j=v_i'+g^2v_j'$, then $v_i-v_i'=g^2(v_j'-v_j)$, which is impossible unless $i=i'$ and $j=j'$. So the element of $W$ are pairwise distinct. This means that $|W|=N^2 \leq q$, which implies $N \leq \sqrt{q}$. Moreover, $N=\sqrt{q}$ if and only if $W=\F_q$.
\end{proof}

Combing Lemma \ref{triv} and Theorem \ref{t4} (regarding Peisert graph as a supergraph of the corresponding quadruple Paley graph), we have the following theorem.

\begin{thm}[{\cite[Theorem 5.1]{KP}}] \label{KP}
Let $q=p^s$, where $p \equiv 3 \pmod 4$ and $s=2k$. If $k$ is odd, then $\omega(P^*_q)= \sqrt{q}$; if $k$ is even, then $\omega(P^*_q) \geq q^{1/4}$ and $\F_{q^{1/4}}$ forms a clique in $P^*_q$.
\end{thm}

In the case that $q$ is not a quartic, the trivial upper bound on the clique number is attained. While in the case that $q$ is a quartic, there remains a huge gap between the best known lower bound $q^{1/4}$ and the best known upper bound $\sqrt{q}$ on the clique number. See \cite[Lemma 3.3.6]{NM} for a discussion on how the tightness of the trivial upper bound would provide a new proof on showing Paley graphs and Peisert graphs are non-isomorphic.%We are interested in whether the trivial upper bound can be attained.

Now we are ready to prove Theorem \ref{maxP*}. Note that it is similar to Corollary \ref{cor4}. Although the connection set of a Peisert graph is not closed under multiplication, we can take advantage of the fact that the maximal clique, with vector space structure, has dimension at most 2.

\begin{proof}[Proof of Theorem \ref{maxP*}]
Let $g$ be a primitive root in $\F_q$. By Theorem \ref{KP}, $\F_{p^r}$ is a clique in $P_q^*$.  Note that each nonzero element in $\F_{p^r}$ has the form $g^{j}$ for some integer $j \equiv 0 \pmod 4$. 

Suppose $\F_{p^r}$ is not a maximal clique in $P^*_q$.
 Then there is $h \notin \F_{p^r}$, such that $\{h\} \cup \F_{p^r}$ forms a larger clique, i.e. for any $a \in \F_{p^r}$, $h-a=g^j$ for some integer $j \equiv 0,1 \pmod 4$. We claim that $\F_{p^r} \oplus h\F_{p^r}$ forms a clique. Let $a,b,c,d \in \F_{p^r}$, we consider the difference
$
(a+hb)-(c+hd)=(a-c)+h(b-d).
$
\begin{itemize}
    \item If $b=d$, then $a-c \in \F_{p^r}$ and $a-c=g^j$ for some $j \equiv 0 \pmod 4$.
    \item If $b \neq d$, then $(a+hb)-(c+hd)=(b-d) \big(h- (b-d)^{-1}(c-a)\big)=g^j \cdot g^k=g^{j+k}$ for some $j \equiv 0 \pmod 4, k \equiv 0,1 \pmod 4$, and thus $j+k \equiv 0,1 \pmod 4$.
\end{itemize}
This shows that $a+hb$ and $c+hd$ are adjacent. Therefore, $C=\F_{p^r} \oplus h\F_{p^r}$ is a clique in $P^*_q$ and $|C|=p^{2r}=\sqrt{q}$. Moreover, $\omega(P^*_q)= \sqrt{q}$ and $C$ is a maximum clique. By Lemma \ref{triv}, $$\F_q=C \oplus g^2C=\F_{p^r} \oplus h\F_{p^r} \oplus g^2\F_{p^r} \oplus g^2h\F_{p^r}.
$$
It follows that $\{1,h,g^2,g^2h\}$ forms a basis of $\F_q$ over $\F_{p^r}$.
\end{proof}

SageMath \cite{Sage} provides powerful packages in finite field arithmetic and graph theory. Using Sage (the code is attached in Appendix A), we find that $\omega(P_{81}^*)=9$ and $\omega(P_{2401}^*)=17$. See also the computational results in \cite[Section 3.4]{NM}. 

We conjecture that $P_{81}^*$ is the only Peisert graph with quartic order, such that the trivial upper bound on the clique number can be attained. A possible explanation that $P_{81}^*$ is exceptional might be the following: in \cite[Lemma 6.7]{WP2}, Peisert showed that $P_{81}^*$ is isomorphic to $G(9^2)$ (defined in \cite[Section 3]{WP2}), which has exceptional automorphism group.

\begin{conj}\label{conjclique}
If $q$ is a power of a prime $p \equiv 3 \pmod 4$  and $q>3$, then the clique number of the Peisert graph of order $q^4$ is strictly less than $q^2$.
\end{conj}

Note that by Theorem \ref{maxP*}, Conjecture \ref{conjclique} implies the following weaker conjecture:

\begin{conj}\label{conjmaxc}
If $q$ is a power of a prime $p \equiv 3 \pmod 4$  and $q>3$, then $\F_{q}$ is a maximal clique in the Peisert graph of order $q^4$.
\end{conj}

It is easy to design a polynomial-time algorithm to check whether $\F_{q}$ is a maximal clique in the Peisert graph of order $q^4$.  Using Sage \cite{Sage} (the code is attached in Appendix A), we verify that $\F_3$ is not a maximal clique in $P_{81}^*$, which implies that $\omega(P_{81}^*)=9$ by Theorem \ref{maxP*} and is consistent with the numerical evidence. We also verify that for $q \in \{7,9,11,19,23,27,31\}$, $\F_{q}$ is a maximal clique in the Peisert graph of order $q^4$. 

\section*{Acknowledgement}
The author would like to thank Joshua Zahl for valuable suggestions, and  Greg Martin, J\'ozsef Solymosi, and Ethan White for helpful discussions. The author also would like to thank the anonymous referees for a careful reading of the draft.

\appendix
\section{Sage code}

Sage code for finding the clique number of a Peisert graph (for example, of order 81):

\begin{lstlisting}
def Peisert(q):
    K.<a> = GF(q, modulus="primitive")
    pows=[a^(4*i+j) for i in [0..(q-1)/4-1] for j in [0..1]]
    return Graph([K, lambda i,j: i != j and i-j in pows])
X = Peisert(81)
X.clique_number()
\end{lstlisting}

Sage code for checking whether $\F_q$ is a maximal clique in the Peisert graph with order $q^4$, where $q$ is a power of a prime $p \equiv 3 \pmod 4$ (for example, $q=23$):

\begin{lstlisting}
q=23
K.<g> = GF(q^4, modulus="primitive")
m=(q^4-1)/4-1
pows=[g^(4*i+j) for i in [0..m] for j in [0..1]]
count=0
t=(q^4-1)/(q-1)
gg=g^t
for i in [0..m]:
    for j in [0..1]:
        s=0
        h=g^(4*i+j)
        for k in [1..q-1]:
           if (gg^k-h) in pows: 
            s=s+1
           else:
            break
        if s<q-1:
            count=count+1
if (count==(q^4-1)/2):
    print("maximal")
else:
    print("not maximal")
\end{lstlisting}

\end{document}